\documentclass[12pt,a4paper,twoside]{article}

\title{\sc On Maxwell's and Poincar\'e's Constants}
\def\shorttitle{On Maxwell's and Poincar\'e's Constants}
\def\pauthor{Dirk Pauly}

\def\mylabelonoff{off}
\def\allowdisbrk{no}

\usepackage{a4,exscale,ifthen,amsfonts,amssymb,amsmath,amscd,graphicx}
\usepackage[english]{babel}
\usepackage[mathscr]{eucal}

\author{{\sf\pauthor}}
\numberwithin{equation}{section}

\newenvironment{acknow}{{\vspace*{1cm}\noindent\bf Acknowledgements }}{}
\newcommand{\bewboxw}{\mbox{}\hfill $\square$ \\}

\newenvironment{proof}{{\noindent\bf Proof }}{\bewboxw}

\newcommand{\keywords}[1]{{\noindent\bf Key Words }#1}

\ifthenelse{\equal{\mylabelonoff}{on}}
{\newcommand{\mylabel}[1]{\label{#1}\fbox{{\rm #1}}}}{\newcommand{\mylabel}[1]{\label{#1}\makebox[0mm][]{}}}
\ifthenelse{\equal{\allowdisbrk}{yes}}
{\allowdisplaybreaks}{}

\usepackage[mathscr]{eucal}
\usepackage{color}

\newcommand{\Abb}[5]{\begin{array}{ccccc}#1&:&#2&\longrightarrow&#3\\{}&{}&#4&\longmapsto&#5\end{array}}
\newcommand{\ds}{\displaystyle}
\newcommand{\ol}{\overline}
\newcommand{\ul}{\underline}

\newcommand{\nz}{\mathbb{N}}

\newcommand{\rz}{\mathbb{R}}

\newcommand{\rt}{\rz^3}
\newcommand{\rttt}{\rz^{3\times3}}

\newcommand{\rN}{\rz^N}

\DeclareMathOperator{\p}{\partial}

\newcommand{\na}{\nabla}

\DeclareMathOperator{\grad}{grad}

\DeclareMathOperator{\rot}{rot}

\renewcommand{\div}{\operatorname{div}}

\newcommand{\eps}{\varepsilon}
\newcommand{\epsmo}{\eps^{-1}}

\newcommand{\epsu}{\ul{\eps}}
\newcommand{\epso}{\ol{\eps}}
\newcommand{\epsh}{\hat{\eps}}

\newcommand{\om}{\Omega}

\newcommand{\pom}{\p\!\om}

\newcommand{\ga}{\Gamma}

\newcommand{\equi}{\Leftrightarrow}

\newcommand{\qequi}{\quad\equi\quad}

\newcommand{\zvec}[2]{\begin{bmatrix}#1\\#2\end{bmatrix}}

\newcommand{\zmat}[4]{\begin{bmatrix}#1&#2\\#3&#4\end{bmatrix}}

\DeclareMathOperator{\id}{id}

\def\set#1#2{\{#1\,:\,#2\}}

\newcommand{\cp}{c_{\mathtt{p}}}
\newcommand{\cpc}{c_{\mathtt{p},\circ}}

\newcommand{\cm}{c_{\mathtt{m}}}

\newcommand{\cmt}{c_{\mathtt{m,t}}}
\newcommand{\cmteps}{c_{\mathtt{m,t},\eps}}

\newcommand{\cmn}{c_{\mathtt{m,n}}}
\newcommand{\cmneps}{c_{\mathtt{m,n},\eps}}

\def\cR{\mathcal{R}}

\DeclareMathOperator{\Lebesgue}{\mathsf{L}}
\newcommand{\Lgen}[2]{\Lebesgue^{#1}_{#2}}
\def\Li{\Lgen{\infty}{}}

\def\Lt{\Lgen{2}{}}

\def\Ltom{\Lt(\om)}

\def\Ltepsom{\Lgen{2}{\eps}(\om)}

\DeclareMathOperator{\Sobolev}{\mathsf{H}}
\newcommand{\Hgen}[3]{\overset{#3}{\Sobolev}{}^{#1}_{#2}}
\def\Ho{\Hgen{1}{}{}}

\def\Hoom{\Ho(\om)}

\def\Hoc{\Hgen{1}{}{\circ}}

\def\Hocom{\Hoc(\om)}

\DeclareMathOperator{\Cont}{\mathsf{C}}
\newcommand{\Cgen}[2]{\overset{#2}{\Cont}{}^{#1}}

\def\Cic{\Cgen{\infty}{\circ}}

\def\Cicom{\Cic(\om)}

\newcommand{\Hggen}[3]{\overset{#2}{\Sobolev}(\grad;#3)}

\newcommand{\Hdgen}[3]{\overset{#2}{\Sobolev}(\div_{#1};#3)}

\newcommand{\Hgom}{\Hggen{}{}{\om}}

\newcommand{\Hdom}{\Hdgen{}{}{\om}}

\newcommand{\Hgcom}{\Hggen{}{\circ}{\om}}

\newcommand{\Hdcom}{\Hdgen{}{\circ}{\om}}

\newcommand{\scps}[2]{\langle#1,#2\rangle}

\newcommand{\scpsom}[2]{\scps{#1}{#2}_{\om}}

\newtheorem{lem}{Lemma}

\newtheorem{theo}[lem]{Theorem}

\newtheorem{rem}[lem]{Remark}

\DeclareMathOperator{\rotspace}{\mathsf{R}}
\newcommand{\rotgen}[2]{\overset{#2}{\rotspace}{}_{#1}}
\newcommand{\rotgenom}[2]{\rotgen{#1}{#2}(\om)}
\newcommand{\rom}{\rotgenom{}{}}
\newcommand{\rcom}{\rotgenom{}{\circ}}
\newcommand{\rzom}{\rotgenom{0}{}}
\newcommand{\rczom}{\rotgenom{0}{\circ}}

\newcommand{\rcalom}{\cR(\om)}
\newcommand{\rcalcom}{\overset{\circ}{\cR}(\om)}

\DeclareMathOperator{\divspace}{\mathsf{D}}
\newcommand{\divgen}[2]{\overset{#2}{\divspace}{}_{#1}}
\newcommand{\divgenom}[2]{\divgen{#1}{#2}(\om)}
\newcommand{\dom}{\divgenom{}{}}
\newcommand{\dcom}{\divgenom{}{\circ}}
\newcommand{\dzom}{\divgenom{0}{}}
\newcommand{\dczom}{\divgenom{0}{\circ}}

\newcommand{\hoom}{\Hoom}
\newcommand{\hocom}{\Hocom}

\DeclareMathOperator{\harmonic}{\mathcal{H}}
\newcommand{\harm}[2]{\harmonic^{#1}_{#2}}
\newcommand{\harmom}[2]{\harm{#1}{#2}(\om)}
\newcommand{\harmdom}{\harmom{}{\mathtt{D}}}
\newcommand{\harmdepsom}{\harmom{}{\mathtt{D},\eps}}

\newcommand{\harmdidom}{\harmom{}{\mathtt{D},\id}}
\newcommand{\harmnom}{\harmom{}{\mathtt{N}}}
\newcommand{\harmnepsom}{\harmom{}{\mathtt{N},\eps}}

\newcommand{\pid}{\pi_{\mathtt{D}}}
\newcommand{\pin}{\pi_{\mathtt{N}}}

\newcommand{\scpsepsom}[2]{\scps{#1}{#2}_{\om,\eps}}

\newcommand{\normo}[1]{\left|#1\right|} 
\newcommand{\normos}[1]{|#1|} 

\newcommand{\normosom}[1]{\normos{#1}_{\om}} 

\newcommand{\normosepsom}[1]{\normos{#1}_{\om,\eps}}

\newcommand{\tcomp}[1]{#1_{\mathtt{t}}}
\newcommand{\ncomp}[1]{#1_{\mathtt{n}}}
\newcommand{\ttr}[2]{\tcomp{#1}|_{#2}}
\newcommand{\ntr}[2]{\ncomp{#1}|_{#2}}
\newcommand{\ttrG}[1]{\ttr{#1}{\Gamma}}
\newcommand{\ntrG}[1]{\ntr{#1}{\Gamma}}
\newcommand{\trbd}[2]{{#1}|_{#2}}
\newcommand{\trG}[1]{\trbd{#1}{\Gamma}}

\DeclareMathOperator{\diam}{diam}
\newcommand{\diamom}{\diam(\om)}

\begin{document}

\maketitle{}

\begin{center}
{\tt Dedicated to Sergey Igorevich Repin
on the occasion of his 60th birthday}
\end{center}

\begin{abstract}
We prove that for bounded and convex domains in three dimensions, 
the Maxwell constants are bounded from below and above by Friedrichs' and Poincar\'e's constants.
In other words, the second Maxwell eigenvalues
lie between the square roots of the second Neumann-Laplace 
and the first Dirichlet-Laplace eigenvalue.\\
\end{abstract}

\tableofcontents

\section{Introduction}

It is well known that, e.g., for bounded Lipschitz domains $\om\subset\rt$,
a square integrable vector field $v$ having square integrable divergence $\div v$
and square integrable rotation vector field $\rot v$ as well as
vanishing tangential or normal component on the boundary $\ga$, i.e,
$\ttrG{v}=0$ resp.~$\ntrG{v}=0$, satisfies the Maxwell estimate
\begin{align}
\mylabel{maxestintro}
\int_{\om}|v|^2\leq\cm^2\int_{\om}\big(|\rot v|^2+|\div v|^2\big),
\end{align}
if in addition $v$ is perpendicular to the so called Dirichlet or Neumann fields, i.e.,
$$\int_{\om}v\cdot w=0\quad\forall\,w\in\harmom{}{}{},$$
where
$$\harmom{}{}{}=
\begin{cases}
\harmdom:=\set{w\in\Ltom}{\rot w=0,\,\div w=0,\,\ttrG{w}=0},&
\text{if }\ttrG{v}=0,\\
\harmnom:=\set{w\in\Ltom}{\rot w=0,\,\div w=0,\,\ntrG{w}=0},&
\text{if }\ntrG{v}=0
\end{cases}$$
holds. Here, $\cm$ is a positive constant independent of $v$,
which will be called Maxwell constant.
See, e.g., \cite{picardpotential,picardboundaryelectro,leisbook,webercompmax}.
We note that \eqref{maxestintro} is valid in much more general situations
modulo some more or less obvious modifications, such as
for mixed boundary conditions, in unbounded (like exterior) domains, 
in domains $\om\subset\rN$, on $N$-dimensional Riemannian manifolds, 
for differential forms or in the case of inhomogeneous media.
See, e.g.,\cite{jochmanncompembmaxmixbc,paulystatic,paulydeco,
picardboundaryelectro,picardcomimb,picarddeco,webercompmax,weckmax}.

So far, to the best of the author's knowledge, general bounds for the Maxwell constants $\cm$ are unknown.
On the other hand, at least estimates for $\cm$ from above are very important
from the point of view of applications, such as preconditioning
or a priori and a posteriori error estimation for numerical methods.

In this contribution we will prove that for 
bounded and convex domains $\om\subset\rt$
\begin{align}
\mylabel{cmcpintro}
\cpc\leq\cm\leq\cp\leq\diamom/\pi
\end{align}
holds true, where $0<\cpc<\cp$ are the Poincar\'e constants, 
such that for all square integrable functions $u$ 
having square integrable gradient $\na u$
$$\int_{\om}|u|^2\leq\cpc^2\int_{\om}|\na u|^2\quad\text{resp.~}\quad
\int_{\om}|u|^2\leq\cp^2\int_{\om}|\na u|^2$$
holds, if $\trG{u}=0$ resp.~$\ds\int_{\om}u=0$. 
While the result \eqref{cmcpintro} is already well known in two dimensions,
even for general Lipschitz domains $\om\subset\rz^{2}$
(except of the last inequality), it is new in three dimensions.
We note that the last inequality in \eqref{cmcpintro}
has been proved in the famous paper of Payne and Weinberger \cite{payneweinbergerpoincareconvex},
where also the optimality of the estimate was shown.
A small mistake in this paper has been corrected later in \cite{bebendorfpoincareconvex}.
We will prove the crucial and from the point of view of applications most interesting inequality
$\cm\leq\cp$ also for polyhedral domains in $\rt$,
which might not be convex but still
allow the $\hoom$-regularity for solutions of Maxwell's equations.
We will give a general result for non-smooth and inhomogeneous, anisotropic media as well,
and even a refinement of \eqref{cmcpintro}.
Let us note that our methods are only based on elementary calculations.

\section{Preliminaries}

Throughout this paper
let $\om\subset\rt$ be a bounded Lipschitz domain.
Many of our results hold true under weaker assumptions on the regularity
of the boundary $\ga:=\pom$. Essentially we need 
the compact embeddings \eqref{allcompactembrellich}-\eqref{allcompactembmaxn} to hold. 
We will use the standard Lebesgue spaces
$\Ltom$ of square integrable functions or vector 
(or even tensor) fields
equipped with the usual $\Ltom$-scalar product $\scpsom{\,\cdot\,}{\,\cdot\,}$
and $\Ltom$-norm $\normosom{\,\cdot\,}$.
Moreover, we will work with the standard $\Ltom$-Sobolev spaces
for the gradient $\grad=\na$, the rotation $\rot=\na\times$
and the divergence $\div=\na\cdot$ denoted by
\begin{align*}
\hoom&:=\Hgom,&\hocom&:=\Hgcom:=\ol{\Cicom}^{\hoom},\\
\dom&:=\Hdom,&\dcom&:=\Hdcom:=\ol{\Cicom}^{\dom},\\
\rom&:=\Hgen{}{}{}(\rot;\om),&\rcom&:=\Hgen{}{}{\circ}(\rot;\om):=\ol{\Cicom}^{\rom}.
\end{align*}
In the latter three Hilbert spaces the classical homogeneous scalar, normal and tangential
boundary traces are generalized, respectively.
An index zero at the lower right corner of the latter spaces indicates a vanishing derivative, e.g.,
$$\rczom:=\set{E\in\rcom}{\rot E=0},\quad\dzom:=\set{E\in\dom}{\div E=0}.$$
Moreover, we introduce a symmetric, bounded ($\Li$) and uniformly positive definite
matrix field $\eps:\om\to\rttt$ and 
the spaces of (harmonic) Dirichlet and Neumann fields
$$\harmdepsom:=\rczom\cap\eps^{-1}\dzom,\quad\harmnepsom:=\rzom\cap\eps^{-1}\dczom.$$
We will also use the weighted $\eps$-$\Ltom$-scalar product 
$\scpsepsom{\,\cdot\,}{\,\cdot\,}:=\scpsom{\eps\,\cdot\,}{\,\cdot\,}$
and the corresponding induced weighted 
$\eps$-$\Ltom$-norm $\normosepsom{\,\cdot\,}:=\scpsepsom{\,\cdot\,}{\,\cdot\,}^{1/2}$. 
Moreover, $\bot_{\eps}$ denotes orthogonality with respect to the $\eps$-$\Ltom$-scalar product.
If we equip $\Ltom$ with this weighted scalar product we write $\Ltepsom$.
If $\eps$ equals the identity $\id$, we skip it in our notations,
e.g., we write $\bot:=\bot_{\id}$ and $\harmdom:=\harmdidom$.
By the assumptions on $\eps$ we have
\begin{align}
\mylabel{epsuoone}
\exists\,\epsu,\epso>0\quad\forall\,E\in\Ltom\quad
\epsu^{-2}\normosom{E}^2\leq\scpsom{\eps E}{E}\leq\epso^2\normosom{E}^2
\end{align}
and we note $\normosepsom{E}^2=\scpsom{\eps E}{E}=\normosom{\eps^{1/2}E}^2$
as well as $\normosom{\eps E}=\normosepsom{\eps^{1/2}E}$. Thus, for all $E\in\Ltom$
\begin{align}
\mylabel{epsuotwo}
\epsu^{-1}\normosom{E}&\leq\normosepsom{E}\leq\epso\normosom{E},&
\epsu^{-1}\normosepsom{E}&\leq\normosom{\eps E}\leq\epso\normosepsom{E}.
\end{align}
For later purposes let us also define $\epsh:=\max\{\epsu,\epso\}$.

We have the following compact embeddings:
\begin{align}
\mylabel{allcompactembrellich}
\hocom\subset\hoom&\hookrightarrow\Ltom&\text{(Rellich's selection theorem)}&\\
\mylabel{allcompactembmaxt}
\rcom\cap\eps^{-1}\dom&\hookrightarrow\Ltom&\text{(tangential Maxwell compactness property)}&\\
\mylabel{allcompactembmaxn}
\rom\cap\eps^{-1}\dcom&\hookrightarrow\Ltom&\text{(normal Maxwell compactness property)}&
\end{align}

It is well known and easy to prove by standard indirect arguments that
\eqref{allcompactembrellich} implies the Poincar\'e estimates
\begin{align}
\mylabel{poincarehoc}
\exists\,\cpc&>0&\forall\,u&\in\hocom&\normosom{u}&\leq\cpc\normosom{\na u},\\
\mylabel{poincareho}
\exists\,\cp&>0&\forall\,u&\in\hoom\cap\rz^{\bot}&\normosom{u}&\leq\cp\normosom{\na u}.
\end{align}
Furthermore
$$\cpc^2=\frac{1}{\lambda_{1}}<\frac{1}{\mu_{2}}=\cp^2$$
holds, where $\lambda_{1}$ is the first Dirichlet 
and $\mu_{2}$ the second Neumann eigenvalue of the Laplacian.
We even have $0<\mu_{n+1}<\lambda_{n}$ for all $n\in\nz$,
see e.g. \cite{filonovdirneulapeigen} and the literature cited there.

Analogously, \eqref{allcompactembmaxt} implies $\dim\harmdepsom<\infty$\footnote{
$d_{\mathtt{D}}:=\dim\harmdepsom$ is finite and independent of $\eps$.
In particular, $d_{\mathtt{D}}$ depends just on the topology of $\om$.
More precisely, $d_{\mathtt{D}}=\beta_{2}$, the second Betti number of $\om$.
A similar result holds also for the Neumann fields, i.e.,
$d_{\mathtt{N}}:=\dim\harmnepsom=\beta_{1}$.}, 
since the unit ball in $\harmdepsom$ is compact,
and the tangential Maxwell estimate, i.e.,
there exists $\cmteps>0$ such that
\begin{align}
\mylabel{maxestelec}
\forall\,E&\in\rcom\cap\eps^{-1}\dom&
\normosepsom{(1-\pid)E}
&\leq\cmteps\big(\normosom{\rot E}^2+\normosom{\div\eps E}^2\big)^{1/2},
\end{align}
where $\pid:\Ltepsom\to\harmdepsom$ denotes the $\eps$-$\Ltom$-orthogonal projector onto Dirichlet fields.
Similar results hold if one replaces the tangential or electric boundary condition 
by the normal or magnetic one. 
More precisely, \eqref{allcompactembmaxn} implies $\dim\harmnepsom<\infty$ 
and the corresponding normal Maxwell estimate, i.e.,
there exists $\cmneps>0$ such that
\begin{align}
\mylabel{maxestmag}
\forall\,H&\in\rom\cap\eps^{-1}\dcom&
\normosepsom{H-\pin H}
&\leq\cmneps\big(\normosom{\rot H}^2+\normosom{\div\eps H}^2\big)^{1/2},
\end{align}
where $\pin:\Ltepsom\to\harmnepsom$ 
denotes the $\eps$-$\Ltom$-orthogonal projector onto Neumann fields.
We note that $\sqrt{\cmteps^2+1}$ can also be seen as the norm of the
inverse $M^{-1}$ of the corresponding electro static Maxwell operator 
$$\Abb{M}{\rcom\cap\eps^{-1}\dom\cap\harmdepsom^{\bot_{\eps}}}{\rot\rcom\times\Ltom}{E}{(\rot E,\div\eps E)}.$$
The analogous statement holds for $\cmneps$ as well.

The compact embeddings \eqref{allcompactembrellich}-\eqref{allcompactembmaxn} 
hold for more general bounded domains with
weaker regularity of the boundary $\ga$, 
such as domains with cone property, restricted cone property or just $p$-cusp-property.
See, e.g., \cite{amrouchebernardidaugegiraultvectorpot,amroucheciarletciarletweakvectorpot,
picardpotential,picardboundaryelectro,picardcomimb,picarddeco,picardweckwitschxmas,
webercompmax,weckmax,witschremmax,leisbook}.
Note that the Maxwell compactness properties and hence the Maxwell estimates 
hold for mixed boundary conditions as well, 
see \cite{jochmanncompembmaxmixbc,goldshteinmitreairinamariushodgedecomixedbc,
jakabmitreairinamariusfinensolhodgedeco}.
The boundedness of the underlying domain $\om$ is crucial,
since one has to work in weighted Sobolev spaces in unbounded (like exterior) domains,
see \cite{kuhnpaulyregmax,leistheoem,leisbook,paulytimeharm,paulystatic,
paulydeco,paulyasym,picardpotential,picardweckwitschxmas}.

As always in the theory of Maxwell's equations, we need another crucial tool,
the Helmholtz or Weyl decompositions of vector fields into irrotational and solenoidal vector fields.
We have
\begin{align*}
\Ltepsom
&=\na\hocom\oplus_{\eps}\eps^{-1}\dzom\\
&=\rczom\oplus_{\eps}\eps^{-1}\rot\rom\\
&=\na\hocom\oplus_{\eps}\harmdepsom\oplus_{\eps}\eps^{-1}\rot\rom,\\
\Ltepsom
&=\na\hoom\oplus_{\eps}\eps^{-1}\dczom\\
&=\rzom\oplus_{\eps}\eps^{-1}\rot\rcom\\
&=\na\hoom\oplus_{\eps}\harmnepsom\oplus_{\eps}\eps^{-1}\rot\rcom,
\end{align*}
where $\oplus_{\eps}$ denotes the orthogonal sum 
with respect the latter scalar product, and note
\begin{align*}
\na\hocom
&=\rczom\cap\harmdepsom^{\bot_{\eps}},&
\eps^{-1}\rot\rom
&=\eps^{-1}\dzom\cap\harmdepsom^{\bot_{\eps}},\\
\na\hoom
&=\rzom\cap\harmnepsom^{\bot_{\eps}},&
\eps^{-1}\rot\rcom
&=\eps^{-1}\dczom\cap\harmnepsom^{\bot_{\eps}}.
\end{align*}
Moreover, with
\begin{align*}
\rcalom:=\rom\cap\rot\rcom
&=\rom\cap\dczom\cap\harmnom^{\bot},\\
\rcalcom:=\rcom\cap\rot\rom
&=\rcom\cap\dzom\cap\harmdom^{\bot}
\end{align*}
we see
$$\rot\rom=\rot\rcalom,\quad
\rot\rcom=\rot\rcalcom.$$
Note that all occurring spaces are closed subspaces of $\Ltom$,
which follows immediately by the estimates \eqref{poincarehoc}-\eqref{maxestmag}.
More details about the Helmholtz decompositions can be found e.g. in \cite{leisbook}.

If $\om$ is even convex\footnote{Note that convex domains
are always Lipschitz, see e.g. \cite{grisvardbook}.} we have some simplifications
due to the vanishing of Dirichlet and Neumann fields, i.e.,
$\harmdepsom=\harmnepsom=\{0\}$. Then
\eqref{maxestelec} and \eqref{maxestmag} simplify to
\begin{align}
\mylabel{maxestelecconv}
\forall\,E&\in\rcom\cap\eps^{-1}\dom&
\normosepsom{E}
&\leq\cmteps\big(\normosom{\rot E}^2+\normosom{\div\eps E}^2\big)^{1/2},\\
\mylabel{maxestmagconv}
\forall\,H&\in\rom\cap\eps^{-1}\dcom&
\normosepsom{H}
&\leq\cmneps\big(\normosom{\rot H}^2+\normosom{\div\eps H}^2\big)^{1/2}
\end{align}
and we have
$$\rczom=\na\hocom,\quad\rzom=\na\hoom,\quad\dzom=\rot\rom,\quad\dczom=\rot\rcom$$
as well as the simple Helmholtz decompositions
\begin{align}
\mylabel{helmdecoconvex}
\Ltepsom=\na\hocom\oplus_{\eps}\eps^{-1}\rot\rom,\quad
\Ltepsom=\na\hoom\oplus_{\eps}\eps^{-1}\rot\rcom.
\end{align}

The aim of this paper is to give a computable estimate for the two Maxwell 
constants $\cmteps$ and $\cmneps$.

\section{The Maxwell Estimates}

First, we have an estimate for irrotational fields, which is well known.

\begin{lem}
\mylabel{lemNarbdiv}
For all $E\in\na\hocom\cap\eps^{-1}\dom$ 
and all $H\in\na\hoom\cap\eps^{-1}\dcom$
$$\normosepsom{E}\leq\epsu\cpc\normosom{\div\eps E},\quad
\normosepsom{H}\leq\epsu\cp\normosom{\div\eps H}.$$
\end{lem}

\begin{proof}
Pick a scalar potential $\varphi\in\hocom$ with $E=\na\varphi$. 
Then, by \eqref{poincarehoc}
\begin{align*}
\normosepsom{E}^2
&=\scpsom{\eps E}{\na\varphi}
=-\scpsom{\div\eps E}{\varphi}
\leq\normosom{\div\eps E}\normosom{\varphi}
\leq\cpc\normosom{\div\eps E}\normosom{\na\varphi}\\
&=\cpc\normosom{\div\eps E}\normosom{E}
\leq\epsu\cpc\normosom{\div\eps E}\normosepsom{E}.
\end{align*}
Let $\varphi\in\hoom$ with $H=\na\varphi$ and $\varphi\bot\rz$. 
Since $\eps H\in\dcom$ we obtain as before and by \eqref{poincareho}
\begin{align*}
\normosepsom{H}^2
&=\scpsom{\eps H}{\na\varphi}
=-\scpsom{\div\eps H}{\varphi}
\leq\normosom{\div\eps H}\normosom{\varphi}
\leq\cp\normosom{\div\eps H}\normosom{\na\varphi}\\
&=\cp\normosom{\div\eps H}\normosom{H}
\leq\epsu\cp\normosom{\div\eps H}\normosepsom{H},
\end{align*}
which finishes the proof.
\end{proof}

\begin{rem}
\mylabel{remNarbdiv}
Without any change, Lemma \ref{lemNarbdiv} extends to Lipschitz domains 
$\om\subset\rN$ of arbitrary dimension.
\end{rem}

To get similar estimates for solenoidal vector fields we 
need a crucial lemma from \cite[Theorem 2.17]{amrouchebernardidaugegiraultvectorpot},
see also \cite{saranenineqfried,grisvardbook,giraultraviartbook,costabelcoercbilinMax} 
for related partial results.

\begin{lem}
\mylabel{french}
Let $\om$ be convex
and $E\in\rcom\cap\dom$ or $E\in\rom\cap\dcom$.
Then $E\in\hoom$ and
\begin{align}
\mylabel{frenchformula}
\normosom{\na E}^2\leq\normosom{\rot E}^2+\normosom{\div E}^2.
\end{align}
\end{lem}

We note that for $E\in\hocom$ it is clear that for any domain $\om\subset\rt$
\begin{align}
\mylabel{frenchformulaequal}
\normosom{\na E}^2=\normosom{\rot E}^2+\normosom{\div E}^2
\end{align}
holds since $-\Delta=\rot\rot-\na\div$.
This formula is no longer valid if $E$ has just the tangential
or normal boundary condition but for convex domains 
the inequality \eqref{frenchformula} remains true.

\begin{lem}
\mylabel{lemNThreerot}
Let $\om$ be convex. For all vector fields $E\in\rcom\cap\eps^{-1}\rot\rom$
and all vector fields $H\in\rom\cap\eps^{-1}\rot\rcom$
$$\normosepsom{E}\leq\epso\cp\normosom{\rot E},\quad
\normosepsom{H}\leq\epso\cp\normosom{\rot H}.$$
\end{lem}

\begin{proof}
Since $\eps E\in\rot\rom=\rot\rcalom$ there exists a vector potential field
$\Phi\in\rcalom$ with $\rot\Phi=\eps E$
and $\Phi\in\Hoom$ by Lemma \ref{french} since 
$\rcalom=\rom\cap\dczom$.
Moreover, $\Phi=\rot\Psi$ can be represented by some $\Psi\in\rcom$.
Hence, for any constant vector $a\in\rt$ we have
$\scpsom{\Phi}{a}=\scpsom{\rot\Psi}{a}=0$. 
Thus, $\Phi$ belongs to $\Hoom\cap(\rt)^{\bot}$.
Then, since $E\in\rcom$ and by Lemma \ref{french} we get
\begin{align*}
\normosepsom{E}^2
&=\scpsom{E}{\eps E}
=\scpsom{E}{\rot\Phi}
=\scpsom{\rot E}{\Phi}
\leq\normosom{\rot E}\normosom{\Phi}
\leq\cp\normosom{\rot E}\normosom{\na\Phi}\\
&\leq\cp\normosom{\rot E}\normosom{\rot\Phi}
=\cp\normosom{\rot E}\normosom{\eps E}
\leq\epso\cp\normosom{\rot E}\normosepsom{E}.
\end{align*}
Since $\eps H\in\rot\rcom$ there exists a vector potential
$\Phi\in\rcom$ with $\rot\Phi=\eps H$. 
Using the Helmholtz decomposition $\Ltom=\rzom\oplus\rot\rcom$,
we decompose 
$$\rom\ni H=H_{0}+H_{\rot}\in\rzom\oplus\rcalom.$$
Then, $\rot H_{\rot}=\rot H$ and
again by Lemma \ref{french} we see $H_{\rot}\in\hoom$.
Let $a\in\rt$ such that $H_{\rot}-a\in\hoom\cap(\rt)^{\bot}$.
Since $\Phi\in\rcom$ and 
$\scpsom{\rot\Phi}{H_{0}}=0=\scpsom{\rot\Phi}{a}$
as well as by Lemma \ref{french} we obtain
\begin{align*}
\normosepsom{H}^2
&=\scpsom{\eps H}{H}
=\scpsom{\rot\Phi}{H}
=\scpsom{\rot\Phi}{H_{\rot}-a}
\leq\normosom{\eps H}\normosom{H_{\rot}-a}\\
&\leq\cp\normosom{\eps H}\normosom{\na H_{\rot}}
\leq\epso\cp\normosepsom{H}\normosom{\rot H_{\rot}}
=\epso\cp\normosepsom{H}\normosom{\rot H},
\end{align*}
completing the proof.
\end{proof}

\begin{rem}
\mylabel{lemNThreerottwod}
It is well known that Lemma \ref{lemNThreerot} 
holds in two dimensions for any Lipschitz domain $\om\subset\rz^2$.
This follows immediately from Lemma \ref{lemNarbdiv}
if we take into account that in two dimensions the rotation $\rot$
is given by the divergence $\div$ after $90^{\circ}$-rotation 
of the vector field to which it is applied.
We refer to the appendix for details.
\end{rem}

\begin{theo}
\mylabel{maintheo}
Let $\om$ be convex. Then, for all vector fields $E\in\rcom\cap\eps^{-1}\dom$
and all vector fields $H\in\rom\cap\eps^{-1}\dcom$
$$\normosepsom{E}^2
\leq\epsu^2\cpc^2\normosom{\div\eps E}^2+\epso^2\cp^2\normosom{\rot E}^2,\quad
\normosepsom{H}^2
\leq\epsu^2\cp^2\normosom{\div\eps H}^2+\epso^2\cp^2\normosom{\rot H}^2.$$
Thus, $\cmteps\leq\max\{\epsu\cpc,\epso\cp\}$ and
$$\cmteps,\cmneps\leq\epsh\cp\leq\epsh\diam(\om)/\pi.$$
\end{theo}

\begin{proof}
By the Helmholtz decomposition \eqref{helmdecoconvex} we have
$$\rcom\cap\eps^{-1}\dom\ni
E=E_{\na}+E_{\rot}\in\na\hocom\oplus_{\eps}\eps^{-1}\rot\rom$$
with $E_{\na}\in\na\hocom\cap\eps^{-1}\dom$ 
and $E_{\rot}\in\rcom\cap\eps^{-1}\rot\rom$ as well as
$$\div\eps E_{\na}=\div\eps E,\quad
\rot E_{\rot}=\rot E.$$
By Lemma \ref{lemNarbdiv} and Lemma \ref{lemNThreerot} and orthogonality we obtain
$$\normosepsom{E}^2
=\normosepsom{E_{\na}}^2+\normosepsom{E_{\rot}}^2
\leq\epsu^2\cpc^2\normosom{\div\eps E}^2+\epso^2\cp^2\normosom{\rot E}^2.$$
Similarly we have
$$\rom\cap\eps^{-1}\dcom\ni
H=H_{\na}+H_{\rot}\in\na\hoom\oplus_{\eps}\eps^{-1}\rot\rcom$$
with $H_{\na}\in\na\hoom\cap\eps^{-1}\dcom$ 
and $H_{\rot}\in\rom\cap\eps^{-1}\rot\rcom$ as well as
$$\div\eps H_{\na}=\div\eps H,\quad
\rot H_{\rot}=\rot H.$$
By Lemma \ref{lemNarbdiv} and Lemma \ref{lemNThreerot}
$$\normosepsom{H}^2
=\normosepsom{H_{\na}}^2+\normosepsom{H_{\rot}}^2
\leq\epsu^2\cp^2\normosom{\div\eps H}^2+\epso^2\cp^2\normosom{\rot H}^2,$$
which finishes the proof.
\end{proof}

Lower bounds can be computed even for general domains $\om$:

\begin{theo}
\mylabel{maintheolower}
It holds 
$$\frac{\cpc}{\epsu\epso^2}\leq\cmteps,\quad
\frac{\cp}{\epsu\epso^2}\leq\cmneps.$$
\end{theo}

\begin{proof}
Let $\lambda_{1}$ resp.~$\lambda_{1,\eps}$ be the first Dirichlet eigenvalue
of the negative Laplacian $-\Delta$ resp.~weighted Laplacian $-\div\eps\na$, i.e.,
$$\frac{1}{\cpc^2}
=\lambda_{1}
=\inf_{0\neq u\in\hocom}\frac{\normosom{\na u}^2}{\normosom{u}^2}
\geq\frac{1}{\epso^2}\inf_{0\neq u\in\hocom}\frac{\normosepsom{\na u}^2}{\normosom{u}^2}
=\frac{\lambda_{1,\eps}}{\epso^2}.$$
Hence $\lambda_{1,\eps}\leq(\epso/\cpc)^2$. 
Let $u\in\hocom$ be an eigenfunction to $\lambda_{1,\eps}$.
Note that $u$ satisfies 
$$\forall\,\varphi\in\hocom\quad\scpsom{\eps\na u}{\na\varphi}=\lambda_{1,\eps}\scpsom{u}{\varphi}.$$
Then $0\neq E:=\na u$ belongs to $\na\hocom\cap\epsmo\dom=\rczom\cap\epsmo\dom\cap\harmdepsom^{\bot_{\eps}}$ 
and solves $-\div\eps E=-\div\eps\na u=\lambda_{1,\eps}u$.
By \eqref{maxestelec} and \eqref{poincarehoc} we have
\begin{align*}
\normosepsom{E}
&\leq\cmteps\normosom{\div\eps E}
=\cmteps\lambda_{1,\eps}\normosom{u}
\leq\cmteps\lambda_{1,\eps}\cpc\normosom{\na u}
\leq\frac{\cmteps}{\cpc}\epso^2\epsu\normosepsom{E}
\end{align*}
yielding $\cpc\leq\cmteps\epsu\epso^2$.
Now, we follow the same arguments for the Neumann eigenvalues.
Let $\mu_{2}$ resp.~$\mu_{2,\eps}$ be the second Neumann eigenvalue
of the negative Laplacian $-\Delta$ resp.~weighted Laplacian $-\div\eps\na$, i.e.,
$$\frac{1}{\cp^2}
=\mu_{2}
=\inf_{0\neq u\in\hoom\cap\rz^{\bot}}\frac{\normosom{\na u}^2}{\normosom{u}^2}
\geq\frac{1}{\epso^2}\inf_{0\neq u\in\hoom\cap\rz^{\bot}}\frac{\normosepsom{\na u}^2}{\normosom{u}^2}
=\frac{\mu_{2,\eps}}{\epso^2}.$$
Hence $\mu_{2,\eps}\leq(\epso/\cp)^2$. 
Let $u\in\hoom\cap\rz^{\bot}$ be an eigenfunction to $\mu_{2,\eps}$.
Note that $u$ satisfies 
$$\forall\,\varphi\in\hoom\cap\rz^{\bot}\quad\scpsom{\eps\na u}{\na\varphi}=\mu_{2,\eps}\scpsom{u}{\varphi}$$
and that this relation holds even for all $\varphi\in\hoom$.
Then $0\neq H:=\na u$ belongs to $\na\hoom\cap\epsmo\dcom=\rzom\cap\epsmo\dcom\cap\harmnepsom^{\bot_{\eps}}$ 
and $-\div\eps H=-\div\eps\na u=\mu_{2,\eps}u$ holds.
By \eqref{maxestmag} and \eqref{poincareho} we have
\begin{align*}
\normosepsom{H}
&\leq\cmneps\normosom{\div\eps H}
=\cmneps\mu_{2,\eps}\normosom{u}
\leq\cmneps\mu_{2,\eps}\cp\normosom{\na u}
\leq\frac{\cmneps}{\cp}\epso^2\epsu\normosepsom{H}
\end{align*}
yielding $\cp\leq\cmneps\epsu\epso^2$.
The proof is complete.
\end{proof}

\begin{rem}
\mylabel{remmaintheolower}
The latter proof shows that Theorem \ref{maintheolower} 
extends to any Lipschitz domain $\om\subset\rN$ of arbitrary dimension
with the appropriate changes for the rotation operator.
\end{rem}

Combining Theorems \ref{maintheo} and \ref{maintheolower} we obtain:

\begin{theo}
\mylabel{maintheolowerupper}
Let $\om$ be convex. Then
$$\frac{\cpc}{\epsh^3}\leq\cmteps\leq\epsh\cp,\quad
\frac{\cpc}{\epsh^3}<\frac{\cp}{\epsh^3}\leq\cmneps\leq\epsh\cp$$
and hence
$$\frac{\cpc}{\epsh^3}\leq\cmteps,\cmneps\leq\epsh\cp\leq\epsh\diam(\om)/\pi.$$
If additionally $\eps=\id$, then
$$\cpc\leq\cmt\leq\cmn=\cp\leq\diam(\om)/\pi.$$
\end{theo}

\begin{rem}
\mylabel{polyhedra}
Our results extend also to all possibly non-convex polyhedra which allow the 
$\hoom$-regularity of the Maxwell spaces $\rcom\cap\dom$ and $\rom\cap\dcom$
or to domains whose boundaries consist of combinations of convex boundary parts 
and polygonal parts which allow the $\hoom$-regularity.
Is is shown in \cite[Theorem 4.1]{costabelcoercbilinMax} 
that \eqref{frenchformula}, even \eqref{frenchformulaequal}, 
still holds for all $E\in\hoom\cap\rcom$ or $E\in\hoom\cap\dcom$ 
if $\om$ is a polyhedron\footnote{The crucial point is 
that the unit normal is piecewise constant and hence the curvature is zero.}.
We note that even some non-convex polyhedra admit the $\hoom$-regularity of the Maxwell spaces
depending on the angle of the corners, which are not allowed to by too pointy.
\end{rem}

\begin{rem}
\mylabel{eigenvalues}
\begin{itemize}
\item[\bf(i)] 
We conjecture $\cpc<\cmt<\cmn=\cp$ for convex $\om\subset\rt$.
\item[\bf(ii)] 
We note that by Theorem \ref{maintheolowerupper} we have given a new proof of the estimate 
$$0<\mu_{2}\leq\lambda_{1}$$
for convex $\om\subset\rt$. Moreover, the absolute values of the eigenvalues 
of the different Maxwell operators (tangential or normal boundary condition)
lie between $\sqrt{\mu_{2}}$ and $\sqrt{\lambda_{1}}$.
\end{itemize}
\end{rem}

Finally, we note that in the case $\eps=\id$ we can find some different proofs
for the lower bounds in less general settings. For example, 
if $\om$ has a connected boundary, then $\harmdom=\{0\}$ and hence
\begin{align*}
\frac{1}{\cmt^2}
&=\inf_{0\neq E\in\rcom\cap\dom}\frac{\normosom{\rot E}^2+\normosom{\div E}^2}{\normosom{E}^2}\\
&\leq\inf_{0\neq E\in\hocom}\frac{\normosom{\rot E}^2+\normosom{\div E}^2}{\normosom{E}^2}
=\inf_{0\neq E\in\hocom}\frac{\normosom{\na E}^2}{\normosom{E}^2}
=\frac{1}{\cpc^2}
\end{align*}
giving $\cpc\leq\cmt$.
If $\om$ is simply connected, then $\harmnom=\{0\}$ and hence
\begin{align*}
\frac{1}{\cmn^2}
&=\inf_{0\neq H\in\rom\cap\dcom}\frac{\normosom{\rot H}^2+\normosom{\div H}^2}{\normosom{H}^2}\\
&\leq\inf_{0\neq H\in\hocom}\frac{\normosom{\rot H}^2+\normosom{\div H}^2}{\normosom{H}^2}
=\inf_{0\neq H\in\hocom}\frac{\normosom{\na H}^2}{\normosom{H}^2}
=\frac{1}{\cpc^2}
\end{align*}
yielding $\cpc\leq\cmn$.
Another proof could be like this: Again, we assume that $\ga$ is connected for the tangential case
resp.~that $\om$ is simply connected for the normal case.
Let $u\in\hocom$ and $\xi\in\rt$ with $|\xi|=1$. Then $E:=u\xi\in\hocom\subset\rcom\cap\dcom$
and since there are no Dirichlet resp.~Neumann fields, we get by \eqref{maxestelec} resp.~\eqref{maxestmag} 
and $\rot E=\na u\times\xi$, $\div E=\na u\cdot\xi$
$$\normosom{u}^2
=\normosom{E}^2
\leq\cm^2\big(\normosom{\rot E}^2+\normosom{\div E}^2\big)
=\cm^2\normosom{\na u}^2.$$
Therefore $\cpc\leq\cm$, where $\cm=\cmt$ resp.~$\cm=\cmn$.

\begin{acknow}
The author is deeply indebted to Sergey Repin not only for bringing his attention
to the problem of the Maxwell constants in 3D.
Moreover, the author wants to thank Sebastian Bauer und Karl-Josef Witsch 
for long term fruitful and deep discussions.
Finally, the author thanks the anonymous referee for careful reading and valuable suggestions,
especially concerning the lower bounds.
\end{acknow}

\bibliographystyle{plain} 
\bibliography{/Users/paule/Library/texmf/tex/TeXinput/bibtex/paule}

\appendix
\section{Appendix: The Maxwell Estimates in Two Dimensions}

Finally, we want to note that similar but simpler results hold in two dimensions as well.
More precisely, for $N=2$ the Maxwell constants can be estimated by the Poincar\'e constants
in any bounded Lipschitz domain $\om\subset\rz^2$.
Although this is quite well known, we present the results for convenience and completeness.

As noted before, Lemma \ref{lemNarbdiv} holds in any dimension.
In two dimensions the rotation $\rot$ 
differs from the divergence $\div$ just by a $90^{\circ}$-rotation $R$ given by
$$R:=\zmat{0}{1}{-1}{0},\quad R^2=-\id,\quad R^\top=-R=R^{-1}.$$
The same holds for the co-gradient $\lhd:=\rot^{*}$
(as formal adjoint) and the gradient $\na$.
More precisely, for smooth functions $u$ and smooth vector fields $v$ we have
\begin{align*}
\rot v&=\div Rv=\p_{1}v_{2}-\p_{2}v_{1},&
\lhd u&=R\na u=\zvec{\p_{2}u}{-\p_{1}u},\\
\div v&=-\rot Rv,&
\na u&=-R\lhd u
\end{align*}
and thus also $-\Delta u=-\div\na u=\div RR\na u=\rot\lhd u$. 
For the vector Laplacian we have $-\Delta v=\lhd\rot-\na\div$.
Furthermore,
$$v\in\rom\qequi Rv\in\dom,\qquad v\in\rcom\qequi Rv\in\dcom.$$
The Helmholtz decompositions read
\begin{align*}
\Ltom
&=\na\hocom\oplus_{\eps}\eps^{-1}\dzom\\
&=\rczom\oplus_{\eps}\eps^{-1}\lhd\hoom\\
&=\na\hocom\oplus_{\eps}\harmdepsom\oplus_{\eps}\eps^{-1}\lhd\hoom,\\
\Ltom
&=\na\hoom\oplus_{\eps}\eps^{-1}\dczom\\
&=\rzom\oplus_{\eps}\eps^{-1}\lhd\hocom\\
&=\na\hoom\oplus_{\eps}\harmnepsom\oplus_{\eps}\eps^{-1}\lhd\hocom
\end{align*}
and we note
\begin{align*}
\na\hocom&=\rczom\cap\harmdepsom^{\bot_{\eps}},&
\eps^{-1}\lhd\hoom&=\eps^{-1}\dzom\cap\harmdepsom^{\bot_{\eps}},\\
\na\hoom&=\rzom\cap\harmnepsom^{\bot_{\eps}},&
\eps^{-1}\lhd\hocom&=\eps^{-1}\dczom\cap\harmnepsom^{\bot_{\eps}}.
\end{align*}
We also need the matrix $\eps_{R}:=-R\eps R$, 
which fulfills the same estimates as $\eps$, i.e., 
for all $E\in\Ltom$
$$\epsu^{-2}\normosom{E}^2\leq\scpsom{\eps_{R}E}{E}\leq\epso^2\normosom{E}^2,$$
since $\scpsom{\eps_{R}E}{E}=\scpsom{\eps RE}{RE}$ and $\normosom{RE}=\normosom{E}$.
But then the inverse $\eps_{R}^{-1}$ satisfies for all $E\in\Ltom$
$$\epso^{-2}\normosom{E}^2\leq\scpsom{\eps_{R}^{-1}E}{E}\leq\epsu^2\normosom{E}^2,$$
which immediately follows by \eqref{epsuotwo}, i.e.,
$$\scpsom{\eps_{R}^{-1}E}{E}=\normosom{\eps_{R}^{-1/2}E}^2
\begin{cases}
\leq\epsu^2\scpsom{\eps_{R}\eps_{R}^{-1/2}E}{\eps_{R}^{-1/2}E}=\epsu^2\normosom{E}^2\\
\geq\epso^{-2}\scpsom{\eps_{R}\eps_{R}^{-1/2}E}{\eps_{R}^{-1/2}E}=\epso^{-2}\normosom{E}^2
\end{cases}.$$
Hence, for the inverse matrix $\eps_{R}^{-1}=-R\eps^{-1}R$ 
simply $\epsu$ and $\epso$ has to be exchanged.
Furthermore, we have $\eps_{R}^{\pm1/2}=-R\eps^{\pm1/2}R$.

For the solenoidal fields we have the following:

\begin{lem}
\mylabel{lemNTworot}
For all $E\in\rcom\cap\eps^{-1}\lhd\hoom$
and all $H\in\rom\cap\eps^{-1}\lhd\hocom$
$$\normosepsom{E}\leq\epso\cp\normosom{\rot E},\quad
\normosepsom{H}\leq\epso\cpc\normosom{\rot H}.$$
\end{lem}

\begin{proof}
Since $RE\in\dcom$ and $R\eps E\in\na\hoom$
we have $R\eps E\in\na\hoom\cap\eps_{R}\dcom$.
By Lemma \ref{lemNarbdiv} (interchanging $\epsu$ and $\epso$) we get
\begin{align*}
\normosepsom{E}
&=\normosom{\eps^{1/2}E}
=\normosom{R\eps^{-1/2}\eps E}
=\normosom{\eps_{R}^{-1/2}R\eps E}
=\normo{R\eps E}_{\om,\eps_{R}^{-1}}\\
&\leq\epso\cp\normosom{\div\eps_{R}^{-1}R\eps E}
=\epso\cp\normosom{\rot E}.
\end{align*}
Analogously, as $RH\in\dom$ and $R\eps H\in\na\hocom$
we have $R\eps H\in\na\hocom\cap\eps_{R}\dom$.
Again by Lemma \ref{lemNarbdiv} (and again interchanging $\epsu$ and $\epso$) we get
\begin{align*}
\normosepsom{H}
&=\normosom{\eps^{1/2}H}
=\normosom{R\eps^{-1/2}\eps H}
=\normosom{\eps_{R}^{-1/2}R\eps H}
=\normo{R\eps H}_{\om,\eps_{R}^{-1}}\\
&\leq\epso\cpc\normosom{\div\eps_{R}^{-1}R\eps H}
=\epso\cpc\normosom{\rot H},
\end{align*}
which completes the proof.
\end{proof}

Finally, the main result is proved as
Theorems \ref{maintheo}, \ref{maintheolower} and \ref{maintheolowerupper},
but taking into account that there are now possibly Dirichlet and Neumann fields.

\begin{theo}
\mylabel{theoNTwo}
For all $E\in\rcom\cap\eps^{-1}\dom$ and all
$H\in\rom\cap\eps^{-1}\dcom$
\begin{align*}
\normosepsom{E-\pid E}^2
&\leq\epsu^2\cpc^2\normosom{\div\eps E}^2+\epso^2\cp^2\normosom{\rot E}^2,\\
\normosepsom{H-\pin H}^2
&\leq\epsu^2\cp^2\normosom{\div\eps H}^2+\epso^2\cpc^2\normosom{\rot H}^2.
\end{align*}
Thus 
$$\frac{\cpc}{\epsu\epso^2}\leq\cmteps\leq\max\{\epsu\cpc,\epso\cp\},\quad
\frac{\cpc}{\epsu\epso^2}<\frac{\cp}{\epsu\epso^2}\leq\cmneps\leq\max\{\epsu\cp,\epso\cpc\}$$
and hence
$$\frac{\cpc}{\epsu\epso^2}\leq\cmteps,\cmneps\leq\epsh\cp.$$
For $\eps=\id$ it holds 
$$\cpc\leq\cmt\leq\cmn=\cp$$
and if additionally $\om$ is convex we have $\cp\leq\diam(\om)/\pi$.
\end{theo}

\begin{proof}
Using the Helmholtz decomposition we have
$$\rcom\cap\eps^{-1}\dom\cap\harmdepsom^{\bot_{\eps}}\ni
E-\pid E=E_{\na}+E_{\lhd}\in\na\hocom\oplus_{\eps}\eps^{-1}\lhd\hoom$$
with $E_{\na}\in\na\hocom\cap\eps^{-1}\dom$ 
and $E_{\lhd}\in\rcom\cap\eps^{-1}\lhd\hoom$ as well as
$$\div\eps E_{\na}=\div\eps E,\quad
\rot E_{\lhd}=\rot E.$$
Thus, by Lemma \ref{lemNarbdiv} and Lemma \ref{lemNTworot} as well as orthogonality we obtain
$$\normosepsom{E-\pid E}^2
=\normosepsom{E_{\na}}^2+\normosepsom{E_{\lhd}}^2
\leq\epsu^2\cpc^2\normosom{\div\eps E}^2+\epso^2\cp^2\normosom{\rot E}^2.$$
Analogously, we decompose
$$\rom\cap\eps^{-1}\dcom\cap\harmnepsom^{\bot_{\eps}}\ni
H-\pin H=H_{\na}+H_{\lhd}\in\na\hoom\oplus_{\eps}\eps^{-1}\lhd\hocom$$
with $H_{\na}\in\na\hoom\cap\eps^{-1}\dcom$ 
and $H_{\lhd}\in\rom\cap\eps^{-1}\lhd\hocom$ as well as
$$\div\eps H_{\na}=\div\eps H,\quad
\rot H_{\lhd}=\rot H.$$
As before, by Lemma \ref{lemNarbdiv}, Lemma \ref{lemNTworot} and orthogonality 
we see
$$\normosepsom{H-\pin H}^2
=\normosepsom{H_{\na}}^2+\normosepsom{H_{\lhd}}^2
\leq\epsu^2\cp^2\normosom{\div\eps H}^2+\epso^2\cpc^2\normosom{\rot H},$$
yielding the assertion for the upper bounds.
For the lower bounds we refer to Remark \ref{remmaintheolower}, which completes the proof.
\end{proof}

\end{document}